\documentclass[12pt,a4paper]{article}
\usepackage{amsfonts, amsthm, amsmath, amssymb}
\usepackage[english]{babel}

\theoremstyle{plain}
\newtheorem{theorem}{Theorem}

\newtheorem{proposition}{Proposition}

\theoremstyle{definition}

\theoremstyle{remark}
\newtheorem{remark}{Remark}
\newcommand{\be}{\begin{equation}}
\newcommand{\ee}{\end{equation}}

\title{On generators with infinite entropy}

\author{B.M. Gurevich\thanks{The work is supported in part
by the RFBR grant 14-01-00379 a}\\ MSU and IITP RAS}

\date{}

\begin{document}
\maketitle

\bigskip

\begin{abstract}
Many years ago B.S. Pitskel observed that the metric entropy of the
shift transformation in the sample space of a stationary random
process $X=\{X_n,\,n\in \mathbb Z\}$ with a countable number of
states is equal to the conditional entropy
$H(X_0|X_{-1},X_{-2},\dots)$ if $X$ is a stationary Markov chain (in
which case the above conditional entropy is $H(X_0|X_{-1}))$,
whether the entropy $H(X_0)$ is finite or not, while in general the
statement is not true. In this note we present a class of processes
for which Pitskel's observation holds, despite the fact that no of
these processes is a Markov chain of some order.
\end{abstract}

\section{Introduction}
We use standard terminology, notation and basic facts from entropy
theory of dynamical systems (see, e.g., \cite{KFS}). It is well
known that every ergodic automorphism of a Lebesgue space has a
countable generator (generating partition) and every automorphism
with finite entropy has a finite generator. A generator is very
useful for studying the automorphism. In particular, if an
automorphism has finite entropy, the entropy can be immediately
expressed in terms of every generator with finite entropy. However
sometimes such a generator is unknown, while another generator, good
in all senses but with infinite entropy, is at hand. The following
example of this situation was considered in \cite{P}.

Let $X=\{X_n,\,n\in\mathbb Z\}$ be a discrete time stationary
process whose states form an infinite countable set $A$, and
$\Omega=A^\mathbb Z$ be its sample space equipped with the
corresponding measure $\nu$. Define the shift transformation
$T:\Omega\to\Omega$ by
$$
(T\omega)_i=(\omega)_{i-1},\ \ i\in\mathbb Z.
$$
Let $\alpha$ be the partition of $\Omega$ into one-dimensional
cylinders
$$
C_a:=\{\omega\in\Omega:\omega_0=a\},\ \  a\in A.
$$
Since $\nu$ is a $T$-invariant probability measure and $\alpha$ is a
generator for $T$, we have
\begin{equation}
\label{entropT} h_\nu(T)=h_\nu(T,\alpha)=
H_\nu(\alpha|\vee_{n=1}^{\infty}T^{-n}\alpha),
\end{equation}
whenever $H_\nu(\alpha)<\infty$, while if $H_\nu(\alpha)=\infty$,
this is in general not the case. However, as was observed in
\cite{P}, (\ref{entropT}) holds, provided that $\nu$ is a Markov
measure, even if $H_\nu(\alpha)=\infty$. We remark that in this case
\begin{equation}
\label{EntropCond} H_\nu(\alpha|\vee_{n=1}^{\infty}T^{-n}\alpha)
=H_\nu(\alpha|T^{-1}\alpha).
\end{equation}
The aim of this note is to present a class of non-Markov measures on
$\Omega$ for which the situation is exactly as in the Markov case.
These measures are induced by stationary random processes $\tilde X$
with the same state set $A$. The construction is the following. Let
$X=\{X_n,\,n\in\mathbb Z\}$ be a Markov chain with states $a\in A$,
transition probabilities $p_{a,b}$, $a,b\in A$, and stationary
probabilities $\pi_a$, $a\in A$. Denote the corresponding Markov
measure on $\Omega$ by $\mu$ and assume that
\begin{equation}
\label{cond} -\sum_{a\in A}\pi_a\log\pi_a=\infty,\ \ -\sum_{a,b\in
A}\pi_a\,p_{a,b}\log\,p_{a,b}<\infty
\end{equation}
(see Section 3 below for an explicite example of such a Markov
chain). Consider a function $f_0:A\to\mathbb N$ with
\begin{equation}
\label{condf} \sum_{a\in A}\pi_af_0(a)<\infty
\end{equation}
and define $f(\omega):=f_0(\omega_0)$. One can rewrite (\ref{cond})
and (\ref{condf}) as
\begin{equation*}
\label{cond1} H_\mu(\alpha)=\infty, \ \
H_\mu(\alpha|T^{-1}\alpha)<\infty,\ \ \int_\Omega fd\mu<\infty.
\end{equation*}

We now define what is called a suspension automorphism $\tilde
T=(T,f)$ constructed by $T$ and the `roof' function $f$. This
$\tilde T$ acts in the space
$$
\tilde\Omega:=\{\tilde\omega=(\omega,u):\omega\in\Omega,\ 0\le u\le
f(\omega)\}
$$
and is defined by
$$
\tilde T(\omega,u)=\left\{\begin{aligned}(&\omega,u+1),\ u<f(\omega),\\
&(T\omega,0),\ u=f(\omega).\\ \end{aligned}\right.
$$
It is easy to check that $\tilde T$ preserves the probability
measure $\tilde\mu:=\gamma(\mu\times\kappa)|_{\tilde\Omega}$, where
$\kappa$ is the counting measure on $\mathbb Z_+$ and
$\gamma=1/\int_\Omega fd\mu$.

Denote by $\tilde\alpha$ the partition of $\tilde\Omega$ into atoms
of the form
$$
\tilde C_a:=\{(\omega,u):\omega\in C_a,\ u\le f_0(a)\}, \ \ a\in A,
$$
and let
$$
\tilde X_n(\tilde\omega)=a {\rm \ \ iff\ \ } \tilde
T^n\tilde\omega\in\tilde C_a,\ \ a\in A,\ \ n\in\mathbb Z.
$$
The stationary process $\tilde X:=\{\tilde X_n,\ n\in\mathbb Z\}$ is
what we wanted to construct. Its properties are studied in Section
2; in Section 3 we present an explicite example of a Markov chain
that satisfies conditions (\ref{cond}).

\begin{remark}
\label{non-Markov} It is easy to check that if $f$ is unbounded,
then $X_n$ is not a Markov chain (we mean Markov chains of any
order).
\end{remark}

The above-mentioned result  from \cite{P} (see (\ref{entropT})) was
obtained due to Pitskel's observation that, for a countable alphabet
Markov chain $\{X_n\}$ with $H(X_0)<\infty$ or $H(X_0)=\infty$ and
$h(T)<\infty$, where $T$ is the shift in the sample space of
$\{X_n\}$, the Shannon--McMillan--Breiman (SMB) theorem holds true
in the following form (we use the above notation):
\begin{equation}
\label{S-M-B}
\lim_{n\to\infty}\left|\frac{1}{n}\log\mu(C^n(\omega))\right|=
H_\mu(\alpha|T^{-1}\alpha),
\end{equation}
where $(C^n(\omega)$ is the atom of the partition
$\vee_0^{n-1}T^i\alpha$ that contains $\omega$ and where the
$\mu$-a.s. convergence is meant. By (\ref{entropT}),
(\ref{EntropCond}) the right hand side of (\ref{S-M-B}) can be
replaced by $h_\mu(T)$. We proceed in the opposite direction: first
we use Pitscel's result for $\alpha$ and $\mu$ to prove, for
$\tilde\alpha$ and $\tilde\mu$, a similar fact (Theorem
\ref{TeorEntrop}), after which we prove for them the SMB theorem.

\section{Properties of $\tilde X$ and $\tilde T$}
All properties of $\tilde X$ can be expressed in terms of those of
$\tilde T$ and wise versa. We will study $\tilde T$, because this is
a little more convenient. We will refer to the set
$\{(\omega,u)\in\tilde\Omega:u=k\}$ as to the $k$th level. The $k$th
level with the maximal possible $k$ and with the minimal possible
$k$ ($k=0$), will be called the top level and bottom level,
respectively. The former will be denoted by $L$.
\begin{proposition}
\label{generat} If $T$ is ergodic, then the partition $\tilde\alpha$
is a generator for $\tilde T$.
\end{proposition}
\begin{proof}
For every point $\omega\in\Omega$, we call the sequence of symbols
$$
\mathfrak n(\omega,\alpha,T)=(a_i(\omega),\ i\in\mathbb Z)
$$
the $(\alpha,T)$-name of $\omega$ if $T^i\omega\in C_{a_i(\omega)},\
i\in\mathbb Z$, and call its finite susequence corresponding to $i$
from $i_1$ to $i_2$ the $(\alpha,T)$-subname of $\omega$ from $i_1$
to $i_2$. In a similar way we define $\mathfrak
n(\tilde\omega,\tilde\alpha,\tilde T)$, the $(\tilde\alpha,\tilde
T)$-name of $\tilde\omega\in\tilde\Omega$ and its subnames. To prove
the proposition it suffices to find a set
$\bar\Omega\subset\tilde\Omega$ with $\tilde\mu(\bar\Omega)=1$ such
that no two different points in $\bar\Omega$ have identical
$(\tilde\alpha,\tilde T)$-names.

For $\bar\Omega$ we take the set of points $\tilde\omega=(\omega,u)$
such that the $(\alpha,T)$-name of $\omega$ contains no infinite
tails of identical symbols (neither in $-\infty$ or in $+\infty$).
From ergodicity of $T$ it follows that $\tilde\mu(\bar\Omega)=1$.

We prove that for every $\tilde\omega=(\omega,u)\in\bar\Omega$, the
$(\tilde\alpha,\tilde T)$-name of $\tilde\omega$ determines the
$(\alpha,T)$-name of $\omega$, from which it follows that
$\tilde\alpha$ is a generator for $\tilde T$, because $\alpha$ is a
generator for $T$.

Given $\tilde\omega\in\bar\Omega$, we introduce the jump set
$$
J(\tilde\omega):=\{i\in\mathbb Z:a_i(\tilde\omega)\ne
a_{i-1}(\tilde\omega)\}.
$$
It is clear that if $i\in J(\tilde\omega)$, then $\tilde
T^i\tilde\omega$ lies at the bottom level and
$a_i(\tilde\omega)=a_l(\omega)$ for some $l$ belongs to $\mathfrak
n(\omega,\alpha,T)$. Moreover, if $i'>i$, $i'\in J(\tilde\omega)$,
and no $j$ between $i$ and $i'$ belongs to $J(\tilde\omega)$, then
$i'=i+m'(f_0(a_i(\tilde\omega))+1)$ for some $m'\in\mathbb N$. Hence
all $a_j(\tilde\omega)$ with $j$ of the form
$i+m(f_0(a_i(\tilde\omega))+1)$, $m\in\mathbb Z_+$, $0\le m\le m'$,
belong to $\frak n(\omega,\alpha,T)$ (with other indices), while the
remaining $a_j(\tilde\omega)$ with $j$ between $i$ and $i'$ do not
belong to $\frak n(\omega,\alpha,T)$. Thus we have described $\frak
n(\omega,\alpha,T)$ in terms of $\frak
n(\tilde\omega,\tilde\alpha,\tilde T)$.

\end{proof}

\begin{proposition}
\label{entropy} If conditions (\ref{cond}), (\ref{condf}) are
satisfied, then $H_{\tilde\mu}(\tilde\alpha)=\infty$.
\end{proposition}
\begin{proof}
By definition
\begin{align}
\label{InfEntr} H_{\tilde\mu}(\tilde\alpha)=&-\sum_{a\in
A}\tilde\mu(\tilde C_a)\log\tilde\mu(\tilde C_a)= -\sum_{a\in
A}\gamma\pi_af_0(a)\log[\gamma\pi_af_0(a)] \notag\\
&=-\gamma\sum_{a\in A}\pi_af_0(a)\log\gamma-\gamma\sum_{a\in
A}\pi_af_0(a)\log\pi_a \notag\\ &-\gamma\sum_{a\in A}\pi_af_0(a)\log
f_0(a).
\end{align}
The first sum in the right hand side of (\ref{InfEntr}) is
$-\log\gamma$, the second sum equals $+\infty$ (see (\ref{cond})).
At last, due to the concavity of the function $v\mapsto -v\log v$,
\begin{align*}
-\gamma\sum_{a\in A}\pi_af_0(a)\log f_0(a)&=\gamma\sum_{a\in
A}\pi_a(-f_0(a)\log f_0(a))\\&\le\gamma\sum_{a\in
A}\pi_af_0(a)\log\sum_{a\in A}\pi_af_0(a)=-\log\gamma.
\end{align*}
\end{proof}
The main result of this section is the following
\begin{theorem}
\label{TeorEntrop} If $T$ is ergodic, then
$$
H_{\tilde\mu}(\tilde\alpha|\vee_{i=1}^\infty\tilde
T^{-i}\tilde\alpha)=h_{\tilde\mu}(\tilde T).
$$
\end{theorem}
\begin{proof}
1. Denote
$$
\tilde\alpha_m^n(\tilde T):=\vee_{i=m}^n\tilde T^{_i}\tilde\alpha,\
\ m,n\in\mathbb Z,\ m<n,\ \ \tilde\alpha^-_{\tilde T}:
=\vee_{i=1}^\infty\tilde T^{-i}\tilde\alpha.
$$
Clearly, $\tilde\alpha_{-n}^{-1}(\tilde
T)\nearrow\tilde\alpha^-_{\tilde T}$, and we know that
$H(\tilde\alpha|\tilde\alpha_{-n}^{-1}(\tilde T))\searrow
H(\tilde\alpha|\alpha^-_{\tilde T})$ if $H(\tilde\alpha|\tilde
T^{-1}\tilde\alpha)<\infty$ (see \cite{R}, 5.11). We first show that
this condition is satisfied.

To this end we find, for all $a,b\in A$, the conditional measure
\begin{equation}
\label{cond_prob} \tilde\mu(\tilde C_b|\tilde T^{-1}\tilde
C_a)=\frac{\tilde\mu(\tilde C_b\cap T^{-1}\tilde
C_a)}{\tilde\mu(\tilde T^{-1}\tilde C_a)}.
\end{equation}
By the definition of $T,\ \tilde T$ and $\mu,\ \tilde\mu$, if $b\ne
a$, then
\begin{align}
\label{numerat1} \tilde\mu(\tilde C_b & \cap\tilde T^{-1}\tilde
C_a)=\tilde\mu(\{\tilde \omega\in\tilde C_b:\tilde
T\tilde\omega\in\tilde C_a\}) \notag \\
&=\tilde\mu(\{\tilde \omega=(\omega,u):\omega\in C_b,\,T\omega\in
C_a,\,u=f_0(b)\})=\gamma\pi_ap_{a,b},
\end{align}
while
\begin{align}
\label{numerat2} \tilde\mu(\tilde C_a\cap\tilde T^{-1}\tilde
C_a)&=\tilde\mu(\{\tilde\omega=(\omega,u):\omega\in C_a,\
u<f_0(a)\})\notag \\
& +\tilde\mu(\{\tilde\omega =(\omega,u):\omega,T\omega\in C_a,\
u=f_0(a)\}) \notag \\
& =\gamma(\pi_a(f_0(a)+p_{a,a}).
\end{align}
From (\ref{cond_prob }) -- (\ref{numerat2}) we obtain
$\tilde\mu(\tilde C_b|\tilde T^{-1}\tilde
C_a)=p_{a,b}(1+f_0(a))^{-1}$ if $a\ne b$, and $\tilde\mu(\tilde
C_a|\tilde T^{-1}\tilde C_a)=(p_{a,a}+f_0(a))(1+f_0(a))^{-1}$.

Therefore,
\begin{align}
\label{FinEnt1} H(\tilde\alpha & |\tilde
T^{-1}\tilde\alpha)=-\sum_{a\in A}\mu(\tilde T^{-1}\tilde
C_a)\big[\tilde\mu(\tilde C_a|\tilde T^{-1}\tilde C_a)\log\mu(\tilde
C_a|\tilde T^{-1}\tilde C_a)\notag \\
& +\sum_{b\in A\setminus\{a\}}\tilde\mu(\tilde C_b|\tilde
T^{-1}\tilde
C_a)\log\tilde\mu(\tilde C_b|\tilde T^{-1}\tilde C_a)\big]\notag\\
& =-\gamma\sum_{a\in
A}(f_0(a)+1)\pi_a\left[\frac{f_0(a)+p_{a,a}}{f_0(a)+1}
\log\frac{f_0(a)+p_{a,a}}{f_0(a)+1}\right.\notag \\ &+ \sum_{b\in A
\setminus\{a\}}\left.\frac{p_{a,b}}{f_0(a)+1}
\log\frac{p_{a,b}}{f_0(a)+1}\right].
\end{align}
We open the square brackets in (\ref{FinEnt1}) and estimate each sum
obtained. The first sum is
\begin{align*}
S_1: & ==\gamma\sum_{a\in
A}\pi_a(f_0(a)+1)\frac{f_0(a)+p_{a,a}}{f_0(a)+1}
\log\frac{f_0(a)+p_{a,a}}{f_0(a)+1}\\
& =-\gamma\sum_{a\in
A}\pi_a(f_0(a)+p_{a,a})\log\frac{f_0(a)+p_{a,a}}{f_0(a)+1}.
\end{align*}
Since $f_0(a)\ge 1$ and $p_{a,a}\ge 0$, we have
\begin{equation}
\label{ratio} 1/2\le\frac{f_0(a)+p_{a,a}}{f_0(a)+1}\le 1, \ \ a\in
A,
\end{equation}
so that, by (\ref{condf}), $S_1<\infty$.

The second sum is
\begin{align*}
S_2: & =-\gamma\sum_{a\in A}\sum_{b\in A\setminus\{a\}}
\pi_a(f_0(a)+1)\frac{p_{a,b}}{f_0(a)+1}\log\frac{p_{a,b}}{f_0(a)+1}\\
& =-\gamma\sum_{a\in A}\sum_{b\in
A\setminus\{a\}}\pi_ap_{a,b}\log\frac{p_{a,b}}{f_0(a)+1}\\
& =-\gamma\sum_{a\in A}\sum_{b\in
A\setminus\{a\}}\pi_ap_{a,b}\log p_{a,b}\\
& +\gamma\sum_{a\in A}\sum_{b\in
A\setminus\{a\}}\pi_ap_{a,b}\log(f_0(a)+1)\\
& \le\gamma h_\mu(T)+\gamma\sum_{a\in a}\pi_a(f_0(a)+1)<\infty.
\end{align*}
Here we used Pitskel's result from \cite{P} (because $\mu$ is a
Markov measure) and (\ref{condf}).

Thus the inequality $H(\tilde\alpha|\tilde
T^{-1}\tilde\alpha)<\infty$ is proved.

2. We say that an atom
\begin{equation}
\label{C} \tilde C=\tilde T^{-1}\tilde C_{a_1}\cap\dots\cap\tilde
T^{-n}\tilde C_{a_n},\ \ n\ge 2
\end{equation}
of the partition $\tilde\alpha_{-n}^{-1}(\tilde T)$ is good if there
exists $k<n$ such that $a_{k+1}\ne a_1$. From the definition of
$\bar\Omega$ (see the proof of Proposition \ref{generat}) it follows
that each point of $\bar\Omega$ belongs to a good atom of
$\tilde\alpha_{-n}^{-1}(\tilde T)$, beginning with some $n$. Since
$G_n$, the union of good atoms of $\tilde\alpha_{-n}^{-1}(\tilde
T)$, does not decrease in $n$, we see that
$\lim_{n\to\infty}\tilde\mu(G_n)=1$.

We wish to find $\mu(\tilde C_b|\tilde C)$, where $\tilde C\subset
G_n$ and $b\in A$. Assume that
\begin{equation*}
\label{k,k+1} a_1=\dots=a_k\ne a_{k+1}
\end{equation*}
(see (\ref{C})). If $k$ is not of the form $k=m(f_0(a_1)+1)$ for
some $m\in\mathbb N$, then no $\omega\in\tilde C$ can be at the top
level. Hence $\mu(\tilde C_a|\tilde C)=1$.

Now let $k=m(f_0(a_1)+1)$. We say that such $\tilde C$ is a very
good atom and denote the union of very good atoms by $VG_n$. For
such atom we define by induction a sequence of symbols
$a'(1),\,a'(2),\dots$, where $a'(1)=a_{k+1}$ and if
$a'(1),\dots,a'(m)$ are already defined and $a'(m)=a_i$, then
$a'(m+1):=a_{i+f_0(a_i)+1}$.

All points $\tilde\omega\in\tilde C_b\cap\tilde C$ have identical
$(\tilde\alpha,\tilde T)$-subnames from $0$ to $n$, while all
$\omega$ such that $\tilde\omega=(\omega,f_0(b))\in\tilde
C_b\cap\tilde C$ have identical $(\alpha,T)$-subnames from $0$ to
$n'$, where $n'=\max\{l:f_0(a'(0))+\dots+f_0(a'(l))\le n\}$ (notice
that $a'(0)=b$). It follows from the definitions of the measures
$\mu$ and $\tilde\mu$ that
$$
\tilde\mu(\tilde C_b\cap\tilde
C)=\gamma\pi_{a'(n')}p_{a'(n'),a'(n'-1)}\cdots p_{a'(1),b}.
$$
Since $\tilde\mu(\tilde C)=\sum_{b\in A}\tilde\mu(\tilde
C_b\cap\tilde C)$, we have
\begin{equation}
\label{CondMeas1} \tilde\mu(\tilde C_b|\tilde C)=p_{a_1,b}=p_{a,b},
{\text when\ }\tilde C\subset VG_n\cap\tilde T^{-1}\tilde C_a.
\end{equation}
Hence
\begin{equation}
\label{GoodEnt} H_{\tilde\mu}(\tilde\alpha|\tilde C)=-\sum_{b\in
A}p_{a,b}\log p_{a,b}
\end{equation}
if $\tilde C$ is a very good atom from $\tilde T^{-1}\tilde C_a$,
and the conditional entropy is zero if $\tilde C$ is a good but not
very good atom.

Next consider a bad atom $\tilde C=\cap_{i=1}^n \tilde T^{-i}\tilde
C_a$ for some $a\in A$. We denote it by $\tilde C(a)$. For every
$b\in A$,
\begin{equation}
\label{CbC} \tilde\mu(\tilde C_b\cap\tilde C(a))=\tilde\mu(\tilde
C_b\cap\tilde C(a)\cap L)+\tilde\mu(\tilde C_b\cap\tilde
C(a)\cap(\tilde\Omega\setminus L)).
\end{equation}
The first term in (\ref{CbC}) is
\begin{align}
\label{CbC1} \tilde\mu(\tilde C_b\cap\tilde C(a)&\cap
L) \notag \\
=&\tilde\mu(\{\tilde\omega\in\tilde C_b\cap L:\tilde
T^i\tilde\omega\in\tilde C_a,\
i=1,\dots,n\})=\pi_a(p_{a,a})^zp_{a,b},
\end{align}
where $z=[\frac{n}{f_0(a)+1}]$. The second term vanishes if $b\ne
a$, while
\begin{align}
\label{CbC2} \tilde\mu(\tilde C_a\cap\tilde
C(a)\cap(\tilde\Omega\setminus
L))&=\tilde\mu(\{\tilde\omega\in\tilde C_a\cap(\tilde\Omega\setminus
L):
\tilde T^i\tilde\omega\in\tilde C_a,\ i=1,\dots,n\})\notag\\
&=f_0(a)\pi_a(p_{a,a})^z.
\end{align}

Similarly,
\begin{align}
\label{CbC3} \tilde\mu(\tilde C(a))=\tilde\mu(\tilde C(a)\cap L) &
+\tilde\mu(\tilde C(a)\cap(\tilde\Omega\setminus L))=\sum_{b\in
A}\pi_a(p_{a,a})^zp_{a,b}\notag \\
&+f_0(a)\pi_a(p_{a,a})^z =\pi_a(p_{a,a})^z(1+f_0(a))
\end{align}

From (\ref{CbC})--(\ref{CbC3}) we obtain
\begin{equation}
\label{CondBad}\tilde\mu(\tilde C_b|\tilde C(a))=
\left\{\begin{aligned}& \frac{p_{a,b}}{1+f_0(a)},\ b\ne a,\\
& \frac{p_{a,a}+f_0(a)}{1+f_0(a)},\ b=a.\\
\end{aligned}\right.
\end{equation}
Hence
\begin{align}
\label{BadEnt} H_{\tilde\mu}(\tilde\alpha|\tilde C(a)) &
=-\frac{p_{a,a}+f_0(a)}{1+f_0(a)}\log\frac{p_{a,a}+f_0(a)}{1+f_0(a)}
\notag \\ & -\sum_{b\in
A\setminus\{a\}}\frac{p_{a,b}}{1+f_0(a)}\log\frac{p_{a,b}}{1+f_0(a)}
\end{align}
if $\tilde C(a)\subset\tilde T^{-1}\tilde C_a$ is a bad atom of
$\tilde\alpha_{-n}^{-1}(\tilde T)$.

3. We complete the proof by finding the asymptotics of the sum
\begin{equation}
\label{asympt} \sum_{\tilde C}\mu(\tilde C)H_\mu(\tilde\alpha|\tilde
C)=\sum_{\tilde C\subset G_n}\mu(\tilde C)H_\mu(\tilde\alpha|\tilde
C)+\sum_{\tilde C\subset B_n}\mu(\tilde C)H_\mu(\tilde\alpha|\tilde
C)
\end{equation}
as $n\to\infty$. Consider each of these sums separately.

Since $H_\mu(\tilde\alpha|\tilde C)=0$ if $\tilde C\subset
G_n\setminus VG_n$ and due to (\ref{GoodEnt}), the first sum equals
\begin{equation}
\label{FirstSum} \sum_{\tilde C\subset VG_n}\mu(\tilde
C)H_\mu(\tilde\alpha|\tilde C)=-\sum_{a\in A}\,\sum_{\tilde C\subset
VG_n\cap\,\tilde T^{-1}\tilde C_a}\mu(\tilde C)\sum_{b\in
A}p_{a,b}\ln p_{a,b}.
\end{equation}
Clearly, $VG_n\cap\,\tilde T^{-1}\tilde C_a=L\cap G_n\cap\,\tilde
T^{-1}\tilde C_a$, and since $\tilde\mu(G_n)\to 1$ as $n\to\infty$,
we have
$$
\lim_{n\to\infty}\tilde\mu(VG_n\cap\,\tilde T^{-1}\tilde
C_a)=\tilde\mu(L\cap\,\tilde T^{-1}\tilde C_a)=\sum_{b\in
A}\gamma\pi_ap_{a,b}=\gamma\pi_a,
$$
so that (see (\ref{FirstSum}))
\begin{equation}
\label{main} \lim_{n\to\infty}\sum_{\tilde C\subset VG_n}\mu(\tilde
C)H_\mu(\tilde\alpha|\tilde C)=\sum_{a\in A}\gamma\pi_a\sum_{b\in
A}p_{a,b}\log p_{a,b}=\gamma h_\mu(T).
\end{equation}

Before dealing with the second sum in (\ref{asympt}) we state the
following simple auxiliary assertion. Let $r_i\ge 0,\ s_{ij}\ge 0,\
i,j\in\mathbb N$ and $\sum_{i\in\mathbb N}r_i<\infty$,
$\sup_{i,j}s_{i,j}<\infty$, $\lim_{j\to\infty}s_{i,j}=0$ for all
$i$. Then
\begin{equation}
\label{auxil} \lim_{j\to\infty}\sum_{i\in\mathbb N}r_is_{i,j}=0.
\end{equation}

For $\tilde C\subset B_n\cap\tilde T^{-1}\tilde C_a$ we write
$\tilde C(a)$. By (\ref{BadEnt}), for each $a\in A$,
$$
H_{\tilde\mu}(\tilde\alpha|\tilde
C(a))=-\frac{p_{a,a}+f_0(a)}{1+f_0(a)}\log\frac{p_{a,a}+f_0(a)}
{1+f_0(a)}-\sum_{b\in A\setminus
{a}}\frac{p_{a,b}}{1+f_0(a)}\log\frac{p_{a,b}}{1+f_0(a)},
$$
so that
\begin{align}
\label{BadSum} \sum_{a\in A}\tilde\mu(\tilde
C(a))H_{\tilde\mu}(\tilde\alpha|\tilde C(a))=-\sum_{a\in
A}\pi_a(1+f_0(a))(p_{a,a})^z\frac{p_{a,a}+f_0(a)}{1+f_0(a)}\notag \\
-\sum_{a\in A}\sum_{b\in A\setminus\{a\}}
\pi_a(1+f_0(a))(p_{a,a})^z\frac{p_{a,b}}{1+f_0(a)}
\log\frac{p_{a,b}}{1+f_0(a)}.
\end{align}
By (\ref{ratio}) the first sum in (\ref{BadSum}) does not exceed
$\sum_{a\in A}\pi_a(1+f_0(a))(p_{a,a})^z\log\,2$. It is clear that
$0\le p_{a,a}<1$ (the Markov chain $X$ is assumed to be ergodic) and
$z=z(a,n)\to\infty$ as $n\to\infty$. We now can use the above
auxiliary assertion by putting $r_a:=\pi_a(1+f_0(a))$ (we identify
$a$ with $i$) and $s_{a,n}:=(p_{a,a})^z$ (here we identify $n$ with
$j$). From this assertion it follows that the first sum in
(\ref{BadSum}) vanishes as $n\to\infty$.

A similar argument shows that the second sum in (\ref{BadSum})
behaves in the same way. As a result we obtain (see (\ref{entropT})
$\lim_{n\to\infty}H_{\tilde\mu}(\tilde\alpha|\tilde\alpha_{-n}^{-1}
(\tilde T))=\gamma H_\mu(T)$. Now it ramains to use (\ref{main}) and
Abramov's formula according to which $h_{\tilde\mu}(\tilde T)=\gamma
h_{\mu}(T)$.
\end{proof}

\begin{theorem}
\label{Breiman} Let $\tilde C^n(\tilde\omega)$ be the atom of the
partition $\vee_1^{n-1}\tilde T^i\tilde\alpha$ that contains
$\tilde\omega$. Then
$$
\lim_{n\to\infty}\left|\frac{1}{n}\log\tilde\mu(\tilde
C^n(\tilde\omega))\right|= h_{\tilde\mu}(\tilde T),
$$
\end{theorem}
\begin{proof}
Examination of a standard proof of the SMB theorem for finite
partition (see, e.g., \cite{B}) shows that this proof applies to
countable infinite partitions, including partitions of infinite
entropy, provided that the following condition is  satisfied (we
state it for our case, using the above notation). For
$\tilde\omega\in\tilde\Omega$, let
$$
g_n(\tilde\omega):=-\log\tilde\mu(\tilde C^0(\tilde\omega)|\tilde
C^n(\tilde\omega)),
$$
where $\tilde C^0(\tilde\omega),\ \tilde C^n(\tilde\omega)$ are the
atoms of $\tilde\alpha$ and $\tilde\alpha_{-n}^n(\tilde T)$,
respectively, containing $\tilde\omega$. It is required that
\begin{equation}
\label{sup} \int_{\tilde\Omega}\sup_n
g_n(\tilde\omega)\tilde\mu(d\tilde\omega)<\infty.
\end{equation}

We establish (\ref{sup}) using some formulas from the above proof of
Theorem \ref{entropT}. Let $\tilde\omega\in\tilde C_b\cap\tilde
T^{-1}\tilde C_a$ for some $a,b\in A$. From (\ref{CondMeas1}),
(\ref{CondBad}) we see that
\begin{equation}
\label{CondBig}\tilde\mu(\tilde C^0(\tilde\omega)|\tilde
C^n(\tilde\omega))=
\left\{\begin{aligned}& p_{a,b},\
\tilde C^n(\tilde\omega)\subset VG_n,\\
& 1,\ \tilde C^n(\tilde\omega)\subset G_n\setminus VG_n,\ b=a,\\
& 0,\ \tilde C^n(\tilde\omega)\subset G_n\setminus VG_n,\ b\ne a,\\
& \frac{p_{a,b}}{1+f_0(a)},\
\tilde C^n(\tilde\omega)\subset B_n,\ b\ne a,\\
& \frac{p_{a,a}+f_0(a)}{1+f_0(a)},\
\tilde C^n(\tilde\omega)\subset B_n,\ b=a.\\
\end{aligned}\right.\\
\end{equation}
It is evident that if $b\ne a$, then
$$
\tilde\mu(\{\tilde\omega\in\tilde C_b\cap\tilde T^{-1}\tilde C_a:
\tilde C^n(\tilde\omega)\subset G_n\setminus VG_n\})=0.
$$
Hence, on a set of $\tilde\mu$-measure $1$,
$$
g_n(\tilde\omega)\le -\log
p_{a,b}-\log\frac{p_{a,b}}{1+f_0(a)}-
\log\frac{p_{a,a}+f_0(a)}{1+f_0(a)}.
$$
Since the right hand side of tis inequality does not depend on $n$,
one can replace its left hand side by $\sup_n g_n(\tilde\omega)$.
From this we obtain (see (\ref{ratio}))
\begin{align*}
\int_{\tilde\Omega}\sup_n g_n(\tilde\omega)\tilde\mu(d\tilde\omega)
& \le-\sum_{a\in A}\sum_{b\in A}(\pi_ap_{a,b}\log
p_{a,b}+\pi_ap_{a,b}\log p_{a,b})
\notag\\
& +\sum_{a\in A}\pi_a\log(1+f_0(a))- \sum_{a\in
A}\pi_a\log\frac{p_{a,a}+f_0(a)}{1+f_0(a)}\notag\\
&\le 2h_\mu(T)+\int_\Omega f(\omega)\mu(d\omega)+\log\,2<\infty.
\end{align*}
\end{proof}

\section{An example}
In this section we construct a Markov chain with an infinite
countable state set $A$, transition probabilities $p_{a,b}$, $a,b\in
A$, and stationary probabilities $\pi_a$, $a\in A$, that satisfy
conditions (\ref{cond}).

We take $A=\mathbb N$ and $a=i$. Let $p_{i,j}$ be of the form
$p_{i,1}=p_i$, $p_{i,i+1}=q_i$ and $p_{i,j}=0$ for $j\in\mathbb
N\setminus\{1,i+1\}$, where $p_i,q_i$ such that $p_i+q_i=1$ are to
be picked up.

Denote the stationary distribution by $\mathbf\pi$ and the
transition matrix by $\mathcal P$. By solving the equation
$\mathbf\pi\mathcal P=\mathbf\pi$ we obtain
$$
\pi_1=Q^{-1},\ \ \pi_n=Q^{-1}\Pi_{i=1}^{n-1}q_i,\ \
Q=1+\sum_{n=2}^\infty q_i.
$$
Hence $Q$ must be finite.

If we put
$$
q_1:=\frac{1}{2\log^2 2},\ \
q_n:=\frac{1}{(n+1)\log^2(n+1)\Pi_{i=1}^{n-1}q_i},\ n\ge 2,
$$
it is easy to show by induction that
$$
\Pi_{i=1}^{n-1}q_i=\frac{1}{n\log^2n},$$ so that $Q<\infty$.

On the other hand,
$$
-\pi_n\log\pi_n=
-Q^{-1}\Pi_{i=1}^{n-1}q_i\log(Q^{-1}\Pi_{i=1}^{n-1}q_i)
$$
so that
$$
-\sum_{n=1}^\infty\pi_n\log\pi_n=
Q^{-1}\sum_{n=2}^\infty\frac{1}{n\log\,n}+S,
$$
where $S$ is the sum of a converging series. So we see that the
first condition in (\ref{cond}) is satisfied.

Finally, the structure of the transition matrix $\mathcal P$ implies
that
\begin{align*}
\label{final}-&\sum_{k,l\in\mathbb N}\pi_kp_{k,l}\log\,p_{k,l}
=-\sum_{k\in\mathbb
N}\pi_k(p_{k,1}\log\,p_{k,1}+p_{k,k+1}\log\,p_{k,k+1})\\ \le &
\log2\sum_{k\in\mathbb N}\pi_k<\infty.
\end{align*}
Thus the second condition in (\ref{cond}) is also satisfied.

\begin{remark}
\label{function} 1. It is easy to understand that the stationary
Markov chain with the just constructed paprameters induces an
ergodic shift $T$ in its sample space.

2. Clearly there exists an unbounded function $f_0:A\to\mathbb N$
satisfying (\ref{condf}).
\end{remark}

\end{document}